\documentclass[a4paper,11pt]{amsart}
\usepackage{mathrsfs,amsmath,amssymb,amsthm,enumerate}
\usepackage{color}

\theoremstyle{plain}
\newtheorem{theorem}{Theorem}[section]
\newtheorem{proposition}[theorem]{Proposition}
\newtheorem{lemma}[theorem]{Lemma}
\newtheorem{corollary}[theorem]{Corollary}
\numberwithin{equation}{section}
\newtheorem*{maintheorem*}{Main Theorem}

\theoremstyle{definition}
\newtheorem{definition}[theorem]{Definition}
\newtheorem{remark}[theorem]{Remark}
\newtheorem{example}[theorem]{Example}
\newtheorem{question}[theorem]{Question}
\newtheorem{conjecture}[theorem]{Conjecture}

\newcommand{\C}{\mathbb{C}}
\newcommand{\Q}{\mathbb{Q}}
\newcommand{\R}{\mathbb{R}}
\newcommand{\Z}{\mathbb{Z}}

\newcommand{\cI}{\mathcal{I}}
\newcommand{\bI}{\mathbb{I}}
\newcommand{\bA}{\mathbb{A}}
\newcommand{\bB}{\mathbb{B}}

\newcommand{\cZ}{\mathcal{Z}}
\newcommand{\RZ}{\mathbb{R}\mathcal{Z}}

\newcommand{\cC}{\mathcal{C}}
\newcommand{\cR}{\mathcal{R}}

\newcommand{\e}{\mathbf{e}}

\newcommand{\balpha}{\boldsymbol{\alpha}}

\newcommand{\Hom}{\operatorname{Hom}}
\DeclareMathOperator{\row}{row}

\begin{document}
\title[Cohomology ring of the real toric space]{Multiplication structure of the cohomology ring of real toric spaces}

\author[S.~Choi]{Suyoung Choi}
\address{Department of Mathematics, Ajou University, 206, World cup-ro, Yeongtong-gu, Suwon 16499, Republic of Korea}
\email{schoi@ajou.ac.kr}

\author[H.~Park]{Hanchul Park}
\address{School of Mathematics, Korea Institute for Advanced Study (KIAS), 85 Hoegiro Dongdaemun-gu, Seoul 02455, Republic of Korea}
\email{hpark@kias.re.kr}

%\thanks{The first author was supported by Basic Science Research Program through the National Research Foundation of Korea(NRF) funded by the Ministry of Science, ICT \& Future Planning(NRF-2016R1D1A1A09917654).}

\date{\today}


% 05C30     Enumeration in graph theory
% 05E45     Combinatorial aspects of simplicial complexes
% 13F55     Face and Stanley-Reisner rings; simplicial complexes
% 14M25   	Toric varieties, Newton polyhedra
% 18A10     Graphs, diagram schemes, precategories
% 57S17     Finite transformation groups
% 57S25   	Groups acting on specific manifolds

% 52B11     $n$-dimensional polytopes
% 52B20     Lattice polytopes
% 52B22   	Shellability
% 52B35     Gale and other diagrams
% 52B70   	Polyhedral manifolds

% 55U10     Algebraic topology - Simplicial sets and complexes
% 57N65     Algebraic topology of manifolds


\begin{abstract}
    A real toric space is a topological space which admits a well-behaved $\Z_2^k$-action.
    Real moment-angle complexes and real toric varieties are typical examples of real toric spaces.
    A real toric space is determined by a pair of a simplicial complex $K$ and a characteristic matrix $\Lambda$.
    In this paper, we provide an explicit $R$-cohomology ring formula of a real toric space in terms of $K$ and $\Lambda$, where $R$ is a commutative ring with unity in which $2$ is a unit. 
    Interestingly, it has a natural $(\Z \oplus \row \Lambda)$-grading.
    As corollaries, we compute the cohomology rings of (generalized) real Bott manifolds in terms of binary matroids, and we also provide a criterion for real toric spaces to be cohomology symplectic.
\end{abstract}

\maketitle

\tableofcontents

\section{Introduction}
During the last half century, the topology of topological spaces admitting nice torus symmetries has been one of the most important problems in toric geometry and toric topology.
In 1970s, the cohomology of smooth complete toric varieties is computed by Jurkiewicz \cite{Jurkiewicz1980} (for projective case) and Danilov \cite{Danilov1978} (for general case).
Later, their topological generalization, recently known as \emph{quasitoric manifolds}, is also studied in \cite{Davis-Januszkiewicz1991}.
Interestingly, the cohomology ring of such manifolds can be beautifully represented as the quotient of a polynomial ring.
It should be mentioned that those toric spaces are all obtainable as quotients of the moment-angle complexes, which also play an important role in toric topology.
The cohomology of the moment-angle complex is also well known due to \cite{Buchstaber-Panov2015}.

On the other hand, it has also been studied the real analogue of toric spaces in toric geometry.
One notes that a toric variety admits a complex structure.
Hence, there is the natural involution on the toric variety associated to the conjugation, and the fixed part of this involution is known as a \emph{real toric variety}.
Similarly, each toric space has its real analogue; for instant, a \emph{small cover} and a \emph{real moment-angle complex} are counterparts of a quasitoric manifold and a moment-angle complex, respectively.
They admit well-behaved $\Z_2^k$-actions induced from the torus action on the associated toric spaces.
It is known that real toric varieties and small covers are quotients of the real moment-angle complex.
Motivated by this, the spaces which can be obtained from a real moment-angle complex admitting $\Z_2^m$-action by quotient of the subgroup of $\Z_2^m$ are, recently, called \emph{real toric spaces}.
More precisely, for a simplicial complex $K$ on $m$ vertices,  the \emph{real moment-angle complex} $\RZ_K$ (or \emph{the polyhedral product}
$(\underline{D^1},\underline{S^0})^K$) of $K$ is defined as follows:
\begin{align*}
    \RZ_K &= (\underline{D^1},\underline{S^0})^K \\
          &:= \bigcup_{\sigma\in K} \left\{(x_1,\dotsc,x_m)\in (D^1)^m \mid x_i \in S^0 \text{ when }i\notin \sigma\right\},
\end{align*}
    where $D^1=[0,1]$ is the unit interval and $S^0 = \{0,1\}$ is its boundary.
    There is a canonical $\Z_2^m$-action on $\RZ_K$ which comes from
     the $\Z_2$-action on the pair $(D^1,S^0)$.
    Let $n\le m$.
    A linear map $\Z_2^m \to \Z_2^n$ represented by an $(n \times m)$ $\Z_2$-matrix $\Lambda = \left(\lambda(1)\;\dotsb\;\lambda(m)\right)$ determines a subgroup $\ker\Lambda \subset \Z_2^m$ acting on $\RZ_K$.
    We denote by $M^\R(K,\Lambda)$ the associated real toric space, defined to be $\RZ_K/\ker{\Lambda}$.
    %If $K$ is a polytopal $(n-1)$-sphere then $M_\lambda$ is known as a \emph{small cover}~\cite{DJ91}.
    See \cite{Choi-Kaji-Theriault2017} for details.

Like toric spaces, a real toric space lies in a central position in the realm of toric topology, so the (co)homology of real toric spaces has been a particular interest in toric topology.
Jurkiwiecz \cite{Jurkiewicz1985} and Davis-Januszkiewicz \cite{Davis-Januszkiewicz1991} computed $\Z_2$-cohomology rings of real toric varieties and small covers, respectively.
It is of the form of the quotient of polynomial ring with $\Z_2$-coefficient.
It is not until 2012, in their unpublished paper \cite{ST2012}, Suciu and Trevisan established the formula for the rational cohomology group of a small cover as announced in \cite{Suciu2012}.
It has been confirmed by the authors in \cite{Choi-Park2017_torsion}, in which one can see that the formula can be generalized to real toric spaces and the similar formula holds for even much generalized coefficient than the rational coefficient.
In order to describe the formula, let us prepare some notations.
Let $K$ be a simplicial complex on $[m]:=\{1, \ldots, m\}$.
%Throughout this paper we sometimes use the identification $2^{[m]} \cong \Z_2^m$ by the bijection
%\begin{equation}\label{eqn:rowvectoridentifiction}
%    \{i_1,\dotsc,i_\ell\} \mapsto \mathbf{e}_{i_1}+\dotsc +  \mathbf{e}_{i_\ell},
%\end{equation}
%where $\mathbf{e}_i$ is the $i$th coordinate vector of $\Z_2^m$, $1\le i \le m$. Moreover, this identification is a group isomorphism $(2^{[m]},\triangle) \cong (\Z_2^m,+)$.
We note that there is the natural identification between $\Z_2^m$ and $2^{[m]}$, as we will see details in \eqref{eqn:rowvectoridentifiction}.
For each element $\omega \in \Z_2^m$, we denote by $K_\omega$ the induced subcomplex of $K$ induced by the subset of $[m]$ corresponding to $\omega$.
\begin{theorem}\cite[Theorem~4.6]{Choi-Park2017_torsion}\label{thm:cohogpofsmallcover}
    Let $M = M^\R(K,\Lambda)$ be a real toric space and $R$ a commutative ring in which $2$ is a unit. Then there is an $R$-linear isomorphism
    \begin{equation*}
        H^p(M;R) \cong \bigoplus_{\omega\in \row \Lambda} \widetilde{H}^{p-1}(K_\omega;R).
    \end{equation*}
\end{theorem}
Throughout the paper, we assume that $R$ is a commutative ring in which $2$ is a unit as above and the coefficient ring of the cohomology is $R$. A typical example of $R$ is the ring of rationals $\Q$.

As the next step, it is natural to ask how we can compute its ``cohomology ring''.
However, the multiplication structure of the cohomology of a real toric space, even a real moment-angle complex, is rather intricate and difficult to understand as mentioned in \cite[p.157]{Buchstaber-Panov2015}.
In the paper \cite{Choi-Park2017_torsion} of the authors, they provided one theoretical formulation of the cohomology ring as in \cite[Theorem~4.5]{Choi-Park2017_torsion}.
Although the formula is correct, it may be of little use in practice, unfortunately.
There are two main reasons.
One reason is that it contains a Raynold operation $N$ which leads to huge computation.
The other reason is that we do not know whether the isomorphism in Theorem~\ref{thm:cohogpofsmallcover} satisfies one expected property below or not.
We note that if two elements $\omega_1$ and $\omega_2$ are in $\row \Lambda$, so is $\omega_1 + \omega_2$.
It, thus, is reasonable to expect that for $\alpha \in \widetilde{H}^{p-1}(K_{\omega_1})$ and $\beta \in \widetilde{H}^{q-1}(K_{\omega_2})$, the cup product $\alpha \smile \beta$ is in $\widetilde{H}^{p+q-1}(K_{\omega_1 + \omega_2})$.
However, the formula given in \cite[Theorem~4.5]{Choi-Park2017_torsion} does not guarantee in transparent way that the natural group isomorphism such as in Theorem~\ref{thm:cohogpofsmallcover} satisfies the property.

In the present paper, we will provide fancier cohomology ring formula of real toric space which can resolve both problems what the original form has.
let $R\langle u_1,\dotsc,u_m; t_1,\dotsc,t_m \rangle$ be the differential free $R$-algebra with $2m$ generators such that
\[
    \deg u_i=1,\quad \deg t_i=0,\quad du_i=0,\quad dt_i = u_i
\]
and the differential $d$ satisfies the Leibniz rule $d(ab) = da\cdot b + (-1)^{\deg a}\cdot db$ for any homogeneous elements $a$ and $b$.
We denote by $\cR$ the quotient of $R\langle u_1,\dotsc,u_m; t_1,\dotsc,t_m \rangle$ under the following relations
\begin{multline}\label{eqn:utrelations}
    u_iu_i=0,\quad u_it_i=u_i,\quad t_iu_i=-u_i,\quad  t_it_i = 1,\\
      u_iu_j=-u_ju_i,\quad u_it_j = t_ju_i,\quad t_it_j=t_jt_i,
\end{multline}
for $i,j=1,\dotsc,m$ and $i\ne j$.
Let us use the notation $u_\sigma$ (respectively, $t_\sigma$) for the monomial $u_{i_1}\cdots u_{i_k}$ (respectively, $t_{i_1}\cdots t_{i_k}$) where $\sigma = \{i_1,\dotsc,i_k\}$, $i_1 <\dotsb < i_k$, is a subset of $[m]$.
The \emph{Stanley-Reisner ideal} $\cI$ is the ideal generated by all square-free monomials $u_\sigma$ such that $\sigma$ is not a simplex of $K$.
    We write the quotient algebra $\cR^K := \cR/ \cI$.
For $\omega \subseteq [m]$, let us denote by $\cR^K_\omega$ the $R$-submodule of $\cR^K$ generated by $u_\sigma t_{\omega \setminus \sigma}$ for $\sigma \subseteq \omega \subseteq [m]$ and $\sigma \in K$.

\begin{maintheorem*}\label{thm:maintheorem}
    There are $(\Z \oplus \row \Lambda)$-graded $R$-algebra isomorphisms
    $$
        H^\ast(M) \cong H(\cR^K|_{\row\Lambda},d) \cong \bigoplus_{\omega \in \row\Lambda} \widetilde{H}^{*-1}(K_\omega),
    $$
where $\cR^K|_{\row\Lambda}$ is the subalgebra of $\cR^K$ generated by $u_\sigma t_{\omega \setminus \sigma}$ such that $\omega \in \row\Lambda$, and the product structure on $\bigoplus_{\omega \in \row\Lambda} \widetilde{H}^{*-1}(K_\omega)$ is given by the canonical maps
\[
    \widetilde{H}^{k-1}(K_{\omega_1}) \otimes \widetilde{H}^{\ell-1}(K_{\omega_2}) \to \widetilde{H}^{k+\ell-1}(K_{\omega_1+\omega_2})
\]
which are induced by simplicial maps $K_{\omega_1+\omega_2} \to K_{\omega_1}\star K_{\omega_2}$ when $\star$ denotes the simplicial join.
%Furthermore, $H^\ast(M;R)$ has a $(\Z \oplus \row \Lambda)$-grading structure compatible with the isomorphism in Theorem~\ref{thm:cohogpofsmallcover}.
\end{maintheorem*}
%This theorem indeed allows us to compute and understand the cohomology ring of real toric spaces with rational coefficient as well as $R$.
The main theorem can be viewed as the analogue for real moment-angle complexes and real toric spaces of Theorem~4.5.8 of \cite{Buchstaber-Panov2015}. Note that the product need not be zero when $\omega_1 \cap \omega_2 \ne \varnothing$ unlike that of moment-angle complexes $\cZ_K$. We also remark that this result enables to calculate the cohomology of the polyhedral product of the form $(\underline{CX},\underline{X})^K$ in the sense of Theorem~1.9 of \cite{BBCG-Cupproduct}.

The proof and details of the main theorem will be given in Theorems~\ref{thm:cohoofsmallcover} and~\ref{thm:productondelta}. In Section~\ref{sec:examples} we will provide a few examples about cohomology previously unknown and some related questions.

In Sections~\ref{subsec:real_Bott} and \ref{subsec:gen_real_bott}, we compute the rational cohomology ring of (generalized) real Bott manifolds. A real Bott manifold is one of the most important examples of real toric varieties, and each real toric manifold is determined by an upper triangular $\Z_2$-matrix $A$ whose diagonal entries are zero. Interestingly, we can show in Proposition~\ref{prop:R-cohom_real_bott} that the rational cohomology of the real Bott manifold corresponding to $A$ is completely determined by the binary matroid related to $A$.

In Section~\ref{subsec:cohom_symp}, we discuss the criterion as in Lemma~\ref{lem:cohosymrealtoricspace} for real toric spaces to be cohomology symplectic.
In addition, we give some necessary conditions for real moment-angle complex to be cohomology symplectic.

%\blue{We also have one unexpected remark of the theorem. As we mentioned above, the real moment-angle complex $(\underline{D^1},\underline{S^0})^K$ is a simple and important example of polyhedral product $(\underline{X},\underline{Y})^K$. The family of real moment-angle complexes contains moment-angle complexes, another central object in toric topology, and }

%On the other hand, such as DJ, Suciu-Trevisan, Choi-Parktorsion, Cai-Choi, and so forth. blah blah. By restricting the coefficient ring to $R$, we lose some information of the cohomology of the real moment-angle complex, but we can obtain a further understanding of cohomology of real toric spaces in price of it.

%\red{emphasize as a polyhedral product : The real moment-angle complex is the simplest example of the polyhedral product, and lies in a central position in the realm of polyhedral products; it contains moment-angle complexes, and its cohomology enables to calculate those of many other polyhedral products blah blah, in the sense of Theorem~1.12 of [BBCG, Cup-products].}

%\red{Table for relations of $u$ and $t$: Z, RZ (Cai), RZ (CP)}

\section{Cohomology ring of a real moment-angle complex} \label{sec:cohom_real_moment_angle}
In this section, we study the cohomology ring of a real moment-angle complex using a natural CW structure of the cube $(D^1)^m$.
%Throughout this paper, the coefficient ring $R$ will stand for a commutative ring in which $2$ (twice the multiplicative identity) is a unit unless otherwise mentioned.
We basically follow the arguments of \cite{Cai2017} and \cite{Choi-Park2017_torsion}, but with the basis \eqref{eqn:basischange} which causes huge difference as we can see in Section~\ref{sec:cohoofrealtoric}.

We will use the notation $C_\ast(X)$ and $C^\ast(X)$ for the simplicial or cellular (co)chain complex of $X$ when $X$ is a simplicial complex or a CW complex respectively.
Let us fix a simplicial complex $K$ on the vertex set $[m]=\{1, \ldots, m\}$.
Regarding the interval $D^1 = [0,1]$ as the simplicial complex consisting of two $0$-cells $0,1$ and one $1$-cell $\underline{01}$, the $m$-cube $(D^1)^m$ has a natural CW structure coming from the Cartesian product operation.
More precisely, let $D_i^1\cong [0,1]$ be the $i$th factor of $(D^1)^m = D_1^1\times\dotsb\times D_m^1$ which is a CW complex with two $0$-cells $0_i, 1_i$ and one $1$-cells $\underline{01}_i$.
Then every cell of $(D^1)^m$ is given as
\[
    e_1\times\dotsb\times e_m, \quad e_i=0_i,\;1_i\;\text{or }\underline{01}_i.
\]
For $1 \le i \le m$, the cochain complex $C^\ast(D_i^1)$ is the dual graded $R$-module $\Hom(C_\ast(D_i^1),R) = \langle 0_i^\ast,1_i^\ast,\underline{01}_i^\ast\rangle $, such that $\deg 0_i^\ast  = \deg 1_i^\ast = 0$ and $\deg \underline{01}_i^\ast = 1$ where $e_i^\ast$ is the cochain dual to the cell $e_i$.
Furthermore, there is the (simplicial) cup product $\smile$ making $C^\ast(D_i^1)$ a graded $R$-algebra. Among the nine possible combinations $a \smile b$ when $a,\,b \in \{0^*,1^*,\underline{01}^*\}$, observe that only the following survive:
\begin{equation}\label{eqn:cupprodofinterval}
    0^* \smile 0^* = 0^*,\quad 1^* \smile 1^* = 1^*,\quad 0^* \smile \underline{01}^* = \underline{01}^* \smile 1^* = \underline{01}^*.
\end{equation}

We define the differential $R$-module $B^\ast((D^1)^m)$, which is $C^\ast(D_1^1) \otimes \dotsb \otimes C^\ast(D_m^1)$ with a differential $d$ such that
\[
    d(e_1^\ast\otimes\dotsb\otimes e_m^\ast) = \sum_{i=1}^m(-1)^{\sum_{j=1}^{i-1}\deg e_j^\ast} e_1^\ast\otimes\dotsb\otimes e_{i-1}^\ast\otimes de_i^\ast\otimes e_{i+1}^\ast\otimes\dotsb\otimes e_m^\ast.
\]
Note that $d0_i^\ast = -\underline{01}_i^\ast$, $d1_i^\ast = \underline{01}_i^\ast$, and $d\underline{01}_i^\ast = 0$. Actually, the cup products on the factors of $B^\ast((D^1)^m)$ extend to the whole $R$-module so that
\[
   ( e_1^\ast\otimes\dotsb\otimes e_m^\ast) \smile ( f_1^\ast\otimes\dotsb\otimes f_m^\ast) = (-1)^c \bigotimes_{i=1}^m e_i^\ast \smile f_i^\ast,
\]
where
\[
    c = \sum_{i=1}^m \deg f_i^\ast \sum_{j > i} \deg e_j^\ast.
\]
Refer to (11) of \cite{Cai2017}. Then the extended cup product turns $B^\ast((D^1)^m)$ into a graded $R$-algebra.

The \emph{real moment-angle complex} of $K$, denoted by $\RZ_K$, is defined by
\[
    \RZ_K = \bigcup_{\sigma \in K} \prod_{j=1}^{m} Y_\sigma^j,\quad\text{where}\quad Y_\sigma^j = \left\{
                   \begin{array}{ll}
                     D^1_{j}, & \hbox{if $j \in \sigma$;} \\
                     \partial D^1_j = \{0_j, 1_j\}, & \hbox{otherwise.}
                   \end{array}
                 \right.
\]
It is obvious that $\RZ_K$ is a subcomplex of $(D^1)^m$ as a CW complex.
The differential $R$-algebra $B^\ast(\RZ_K)$ is defined as follows.
Since $B^\ast((D^1)^m)$ can be thought as $\Hom(C_\ast((D^1)^m),R)$ where $C_\ast((D^1)^m)$ is the cellular chain complex of $(D^1)^m$, we put $B^\ast(\RZ_K)  =  \Hom(C_\ast(\RZ_K),R)$, where $C_\ast(\RZ_K)$ is the restriction of $C_\ast((D^1)^m)$ to $\RZ_K$, and inherits the cup product from $B^\ast((D^1)^m)$ as an $R$-subalgebra.
A direct application of Theorem~3.1 of \cite{Cai2017} implies that $B^\ast(\RZ_K)$ is indeed a well-defined differential $R$-algebra and the following holds.

\begin{theorem}\cite[Theorem~5.1]{Cai2017}\label{thm:cohoofrzp}
    There is a graded $R$-algebra isomorphism
    \[
        H^\ast(\RZ_K) \cong H(B^\ast(\RZ_K),d).
    \]
\end{theorem}

Now we perform a ``basis change'' of $C^\ast(D^1)$ by
\begin{equation}\label{eqn:basischange}
    \mathbf{1} = 1^\ast+0^\ast,\quad t = 1^\ast - 0^\ast,\quad u = 2\cdot\underline{01}^\ast.
\end{equation}
This is a genuine basis change, because $2$ is a unit in the coefficient ring $R$. Indeed, one has
\[
    1^\ast = \frac12 ( \mathbf{1} + t),\text{\quad } 0^\ast = \frac12 ( \mathbf{1} - t),\text{\quad and \quad }\underline{01}^\ast = \frac12 u.
\]

From \eqref{eqn:cupprodofinterval}, the following identities are easily checked.
\begin{multline*}
    u \smile u = 0,\quad u \smile t =  u,\quad t \smile u = -u,\quad t \smile t = \mathbf{1},\\
    \mathbf{1} \smile u = u \smile \mathbf{1} = u,\quad \mathbf{1} \smile t = t \smile \mathbf{1} = t,\quad \mathbf{1} \smile \mathbf{1} = \mathbf{1}.
\end{multline*}

%\begin{remark}\label{rem:basischange}
    It should be emphasized that the above basis change \eqref{eqn:basischange} is modified from (17) of \cite{Cai2017}, or (3.1) of \cite{Choi-Park2017_torsion}.
    The original basis change in \cite{Cai2017} and \cite{Choi-Park2017_torsion} is
    \begin{equation}\label{eqn:oldbasis}
        \mathbf{1} = 1^\ast+0^\ast,\; t = 1^\ast,\; u = \underline{01}^\ast.
    \end{equation}
    This modification leads great improvements of computability which will be explained later.
%    which can also be performed for arbitrary coefficient ring $R$ including $\Z$ and $\Z_2$.}
%    , but this is not true for \eqref{eqn:basischange}.% Moreover, a version of Theorem~\ref{thm:cohoofrzp} still holds for arbitrary $R$, but it is not true for small covers as we can see in the next section.
%\end{remark}
%
It is convenient to regard $B^\ast(\RZ_K)$ as a differential $R$-algebra with $2m$ generators $u_1,\dotsc,u_m,t_1,\dotsc,t_m$ such that $\mathbf{1}_1\smile\dotsc\smile\mathbf{1}_m$ is the unique identity.
More precisely, let $R\langle u_1,\dotsc,u_m; t_1,\dotsc,t_m \rangle$ be the differential free $R$-algebra with $2m$ generators such that
\[
    \deg u_i=1,\quad \deg t_i=0,\quad du_i=0,\quad dt_i = u_i
\]
and the differential $d$ satisfies the Leibniz rule $d(ab) = da\cdot b + (-1)^{\deg a}\cdot db$ for any homogeneous elements $a$ and $b$.
We denote by $\cR$ the quotient of $R\langle u_1,\dotsc,u_m; t_1,\dotsc,t_m \rangle$ under the following relations
\begin{multline}\label{eqn:utrelations}
    u_iu_i=0,\quad u_it_i=u_i,\quad t_iu_i=-u_i,\quad  t_it_i = 1,\\
      u_iu_j=-u_ju_i,\quad u_it_j = t_ju_i,\quad t_it_j=t_jt_i,
\end{multline}
for $i,j=1,\dotsc,m$ and $i\ne j$.
Let us use the notation $u_\sigma$ (respectively, $t_\sigma$) for the monomial $u_{i_1}\cdots u_{i_k}$ (respectively, $t_{i_1}\cdots t_{i_k}$) where $\sigma = \{i_1,\dotsc,i_k\}$, $i_1 <\dotsb < i_k$, is a subset of $[m]$.
The \emph{Stanley-Reisner ideal} $\cI$ is the ideal generated by all square-free monomials $u_\sigma$ such that $\sigma$ is not a simplex of $K$.
    We write the quotient algebra $\cR^K := \cR/ \cI$.
Note that, as an $R$-module, $\cR^K $ is freely generated by the square-free monomials $u_\sigma t_{\omega \setminus \sigma}$, where $\sigma \subseteq \omega \subseteq [m]$ and $\sigma \in K$.
Now observe that we can identify the differential $R$-algebras
\[
    \cR \cong B^\ast((D^1)^m) \quad \text{ and }\quad \cR^K \cong B^\ast(\RZ_K).
\]
Then Theorem~\ref{thm:cohoofrzp} can be restated as follows.
\begin{theorem}\cite[Theorem~5.1]{Cai2017}\label{thm:cohoofrzpnew}
    There is a graded $R$-algebra isomorphism
    \[
        H^\ast(\RZ_K) \cong H(\cR^K,d).
    \]
\end{theorem}
For $\omega \subseteq [m]$, let us denote by $\cR^K_\omega$ the $R$-submodule of $\cR^K$ generated by $u_\sigma t_{\omega \setminus \sigma}$ for $\sigma \subseteq \omega \subseteq [m]$ and $\sigma \in K$.
The differential $d$ is preserved in $\cR^K_\omega$ for each $\omega\subseteq [m]$, and thus there is a $(\Z \oplus 2^{[m]})$-grading on the $R$-module $H^\ast(\RZ_K)$
\begin{equation}\label{eqn:bigrading}
    H^{i,\omega}(\RZ_K) \cong H^i(\cR^K_\omega,d),
\end{equation}
where $2^{[m]}$ denotes the power set of $[m]$.
However, as an $R$-algebra, $\cR^K$ is generally not $(\Z \oplus 2^{[m]})$-graded due to the relations $u_it_i = u_i$ and $t_iu_i = -u_i$. In Cai's original settings, one has the union $\cup$
\[
    \smile\colon H^{p,\omega}(\RZ_K)\otimes H^{p',\omega'}(\RZ_K) \to H^{p+p',\omega \cup \omega'}(\RZ_K),
\]
and the union does not give a group structure on $2^{[m]}$.
Nevertheless, as we will see in the next section, it turns out by Theorem~\ref{thm:productondelta} that at cohomology level, $H(\cR^K,d)$ is indeed a $(\Z \oplus 2^{[m]})$-graded $R$-algebra.

Denote by $K_\omega = \{\sigma \in K \mid \sigma \subseteq \omega \}$ the induced subcomplex of $K$ with respect to $\omega$. For each $\omega \subseteq [m]$, observe that there is a bijective cochain map of cochain complexes
\begin{align*}
    f_\omega\colon \cR^K_\omega \;&{\stackrel{\cong}{\longrightarrow}} \; C^\ast(K_\omega) \\
    u_\sigma t_{\omega\setminus \sigma} &\longrightarrow \sigma^\ast,
\end{align*}
where $C^\ast(K_\omega)$ means the simplicial cochain complex of $K_\omega$.
This induces an $R$-linear isomorphism of cohomology
\begin{equation}\label{eqn:pomegahomology}
    H^p(\cR^K_\omega,d) \stackrel{\cong}{\longrightarrow} \widetilde{H}^{p-1} (K_\omega)
\end{equation}
and thus one concludes that there is an $R$-linear isomorphism, well-known as the \emph{Hochster formula},
\begin{equation}\label{eqn:rzphochster}
     H^{p,\ast}(\RZ_K) = \bigoplus_{\omega\subseteq[m]} H^{p,\omega}(\RZ_K) \cong \bigoplus_{\omega\subseteq[m]} \widetilde{H}^{p-1}(K_\omega).
\end{equation}
See {\cite[Proposition~3.3]{Cai2017}} for details.

%\begin{remark}
%    It should be noted that Theorem~\ref{thm:cohoofrzpnew} and \eqref{eqn:rzphochster} in fact hold whenever $2$ is a unit or not in $R$. The condition that $2$ is a unit in $R$ is essential for the real toric spaces as well as Theorem~\ref{thm:productondelta}.
%\end{remark}

The cohomology of the \emph{moment-angle complex} $\cZ_K$ is beautifully described in Section~4.5 of \cite{Buchstaber-Panov2015}. In order to compare the cohomology of $\RZ_K$ and $\cZ_K$, we include a brief explanation of $H^\ast(\cZ_K)$. Let $R\langle u_1,\dotsc,u_m; t_1,\dotsc,t_m \rangle^\C$ be the differential free $R$-algebra with $2m$ generators such that
\[
    \deg u_i=2,\quad \deg t_i=1,\quad du_i=0,\quad dt_i = u_i
\]
and the differential $d$ satisfies the Leibniz rule $d(ab) = da\cdot b + (-1)^{\deg a}\cdot db$ for any homogeneous elements $a$ and $b$. We denote by $\cR^\C$ the quotient of $R\langle u_1,\dotsc,u_m; t_1,\dotsc,t_m \rangle^\C$ under the following relations
\begin{multline}\label{eqn:utrelations_zk}
    u_iu_i=0,\quad u_it_i=0,\quad t_iu_i=0,\quad  t_it_i = 0,\\
      u_iu_j=u_ju_i,\quad u_it_j = t_ju_i,\quad t_it_j= - t_jt_i,
\end{multline}
for $i,j=1,\dotsc,m$ and $i\ne j$. Let the Stanley-Reisner ideal $\cI^\C$ be defined analogously to the $\RZ_K$ case. Then one has a graded $R$-algebra isomorphism
\begin{equation}\label{eqn:cohoofzk}
    H^\ast(\cZ_K) \cong H(\cR^\C/\cI^\C,d).
\end{equation}
\begin{remark}
    Besides the similarity of the two $R$-algebras, there are two main concerns.
    \begin{enumerate}
        \item For $\cZ_K$, \eqref{eqn:cohoofzk} holds for an arbitrary coefficient ring $R$. Indeed, we could choose the basis \eqref{eqn:oldbasis} to obtain the result of \cite{Cai2017} for $H^\ast(\RZ_K;R)$ for arbitrary coefficient.
        \item The difference of \eqref{eqn:utrelations} and \eqref{eqn:utrelations_zk} yields significant contrast of the rings $H^\ast(\RZ_K;R)$ and $H^\ast(\cZ_K;R)$. The analogue of Theorem~\ref{thm:productondelta} still holds for $\cZ_K$, but in that case, the cup product is zero if $\omega \cap \omega' \ne \varnothing$, which is not generally true for $\RZ_K$.
    \end{enumerate}
\end{remark}

\section{Cohomology ring of a real toric space}\label{sec:cohoofrealtoric}
We recall that there is a natural action of $\Z_2^m$ on $\RZ_K \subseteq (D^1)^m$ by
\[
    (g_1,\dotsc,g_m) \cdot (x_1,\dotsc,x_m) = (g_1\cdot x_1, \dotsc, g_m \cdot x_m),
\]
where
\begin{equation}\label{eqn:actiononcube}
    g_i\cdot x_i = \left\{
                   \begin{array}{ll}
                     x_i, & \hbox{if $g_i = 0$;} \\
                     1-x_i, & \hbox{if $g_i = 1$.}
                   \end{array}
                 \right.
\end{equation}
Any subgroup of $\Z_2^m$ can be specified as the kernel of a surjective linear map $\Lambda \colon \Z_2^m \to \Z_2^q$ for some $q\leq m$. When $q = 0$, we put $\Lambda$ be the empty matrix by convention. We denote the $i$th column by $\Lambda(i)$.
If $\Lambda$ satisfies the following condition called the \emph{non-singularity condition}
$$
    \Lambda(i_1), \ldots, \Lambda(i_\ell) \text{ are linearly independent in $\Z_2^q$ if $\{i_1, \ldots, i_\ell\} \in K$},
$$ then $\Lambda$ is called a (mod $2$) \emph{characteristic function} over $K$.
One can check (or see \cite[Lemma~3.1]{Choi-Kaji-Theriault2017}) that the action of $\ker \Lambda$ is free on $\RZ_K$ if and only if $\Lambda$ is a characteristic function over $K$. In this case, $\ker \Lambda$ is isomorphic to $\Z_2^{m-q}$.

For a surjective linear map $\Lambda \colon \Z_2^m \to \Z_2^q$, the quotient space $\RZ_K / \ker \Lambda$ is denoted by $M^\R(K,\Lambda)$ and is called the \emph{real toric space} associated to $(K, \Lambda)$. If $K$ is the boundary of a simplicial $n$-polytope and $\Lambda$ is non-singular, then $M^\R(K,\Lambda)$ is a smooth $n$-manifold. If we add the condition $q=n$, then $M^\R(K,\Lambda)$ is the well-known \emph{small cover}.

Our main tool here is the \emph{transfer homomorphism} for the finite group action. For the transfer homomorphism, refer for example \cite[III.2]{Bredon1972} or \cite[Section~3.G]{Hatcher2002} when the action is free. When a group $\Gamma$ acts on the $R$-algebra $A$, let us recall that $A^\Gamma$ means the subalgebra consisting of elements fixed by the $\Gamma$-action.
\begin{theorem}{(See \cite[Theorem~3.2]{Choi-Park2017_torsion})}\label{thm:transferhomomorphism}
    Let $X$ be a CW complex and $\Gamma$ a finite group acting on $X$. Then there is a graded $R$-algebra isomorphism
    \[
        H^\ast(X/\Gamma;R)\cong H^\ast(X;R)^\Gamma,
    \]
    when $R$ is a commutative ring in which the order $|\Gamma|$ is a unit.
\end{theorem}

We are going to compute the induced action of $\ker\Lambda$ on $H^\ast(\RZ_K) \cong H(\cR^K,d)$. Consider the reflection map $f\colon D^1 \to D^1$ given by $f(x) = 1-x$. Then $f$ induces a map on $C^\ast(D^1)$, again denoted by $f$, such that
\[
    f(\underline{01}^\ast) = -\underline{01}^\ast,\quad f(0^\ast) = 1^\ast, \quad\text{and}\quad f(1^\ast) = 0^\ast
\]
and after change of basis,
\[
    f(u) = -u,\quad f(t) = -t, \quad \text{and}\quad f(\mathbf{1}) = \mathbf{1}.
\]

It should be noted that $f$ preserves $H^{\ast, \omega}(\RZ_K)$ for each $\omega \in 2^{[m]}$.
This is one of the key properties of the basis change \eqref{eqn:basischange}.
It leads the simpler and improved cohomology ring formula of $M^\R(K,\Lambda)$ as in Theorem~\ref{thm:cohoofsmallcover} below rather than one given in \cite{Choi-Park2017_torsion}.

%The most important example will be $\ker \Lambda$.

The following lemmas are useful later in this section.
In this paper we sometimes use the identification $2^{[m]} \cong \Z_2^m$ by the bijection
\begin{equation}\label{eqn:rowvectoridentifiction}
    \{i_1,\dotsc,i_\ell\} \mapsto \mathbf{e}_{i_1}+\dotsc +  \mathbf{e}_{i_\ell},
\end{equation}
where $\mathbf{e}_i$ is the $i$th coordinate vector of $\Z_2^m$, $1\le i \le m$. Moreover, this identification is a group isomorphism $(2^{[m]},\triangle) \cong (\Z_2^m,+)$.

\begin{lemma}\cite[Theorem~4.2]{Choi-Park2017_torsion}\label{lem:rowlambda}
    Let us assume that we have the identification $2^{[m]} \cong \Z_2^m$. Let $\row\Lambda $ be the row space of $\Lambda$ and $\omega$ a vector in $\Z_2^m$. Then $\omega\in\row\Lambda$ if and only if $| \omega \cap g|$ is even for all $g \in \ker \Lambda$.
\end{lemma}
\begin{proof} The lemma is proved by the following observation.
    \begin{align*}
        \omega \in \row \Lambda
        &\Longleftrightarrow \omega \perp \ker \Lambda \\
        &\Longleftrightarrow \omega \cdot g = 0 \text{ for all } g \in \ker \Lambda \\
        &\Longleftrightarrow |\omega \cap g|\text{ is even for all }  g \in \ker \Lambda.
    \end{align*}
\end{proof}

\begin{lemma}\label{lem:actiononri}
    The $\Z_2^m$-action on $\RZ_K$ induces a $\Z_2^m$-action on the $R$-module $\cR^K$ by the following. For a monomial $u_\sigma t_{\omega \setminus \sigma} \in \cR^K$ for $\sigma \subseteq \omega$ and $g \in \Z_2^m$,
    \[
        g \cdot u_\sigma t_{\omega \setminus \sigma} = (-1)^{|\omega \cap g|}u_\sigma t_{\omega \setminus \sigma}.
    \]
\end{lemma}
\begin{proof}
    The proof is obvious since some $u_i$ or $t_i$ in $u_\sigma t_{\omega \setminus \sigma}$ changes its sign whenever $i \in g \cap \omega$.
\end{proof}
One should be cautious that the $\Z_2^m$-action on the $R$-module $\cR^K$ does not preserve the product structure. For instance, $f(ut) = f(u) = -u$, while $f(u)f(t) = (-u)\cdot (-t) = ut = u$. But its induced action on the cohomology $H(\cR^K,d) \cong H^\ast(\RZ_K)$ does preserve the product by the functorial property of the cup product. %\blue{check}

%We regard the characteristic matrix $\Lambda\colon \Z_2^m \to \Z_2^q$ as a $\Z_2$-linear map.

For $A \subseteq \Z_2^m$, a set of vectors in $\Z_2^m$, we denote by $\cR^K|_{A}$ the direct sum
\[
    \cR^K|_{A} = \bigoplus_{\omega \in A} \cR^K_\omega.
\]
\begin{theorem}\label{thm:cohoofsmallcover}
    Let $M = M^\R(K,\Lambda)$ be a real toric space. Then there is a graded $R$-algebra isomorphism
    \begin{equation}\label{eqn:cohoiscohoofr}
        H^\ast(M) = \bigoplus_{\omega \in  \row \Lambda}H^{\ast,\omega}(\RZ_K)  \cong H(\cR^K|_{\row\Lambda},d).
    \end{equation}
\end{theorem}
\begin{proof}
    Theorem~\ref{thm:transferhomomorphism} implies that $H^\ast(M) \cong H^\ast(\RZ_K)^{\ker\Lambda} $ since $|\ker\Lambda| = 2^{m-q}$ is a unit in $R$. By Theorem~\ref{thm:cohoofrzpnew}, it is enough to show that
    \[
        H(\cR^K,d)^{\ker\Lambda} \cong H(\cR^K|_{\row\Lambda},d).
    \]
%    To show that $H(\cR^K,d)^{\ker\Lambda} \supseteq H(\cR^K|_{\row\Lambda},d)$, suppose that we are given a nonzero cohomology class $\alpha \in H^{\ast,\omega}(\RZ_K;R)$ when $\omega \in \row\Lambda$. Let $g$ be an element of $\ker\Lambda$. By Lemma~\ref{lem:actiononri}, one observes that
%    \[
%        g \cdot \alpha = \left\{
%                        \begin{array}{ll}
%                          \alpha, & \hbox{if $g \cap \omega$ has even cardinality, or} \\
%                          -\alpha, & \hbox{if $g \cap \omega$ has odd cardinality.}
%                        \end{array}
%                      \right.
%    \]
%    Note that $\alpha \ne -\alpha$; if $\alpha = -\alpha$, then $\alpha + \alpha = 2\alpha = 0$ and we conclude that $\alpha = 0$ \red{since $2$ is a unit in $R$. Thus,}
%    %by the property of $R$. Now,
%    $\alpha$ is fixed by $\ker \Lambda$ if and only if $\omega \in \row \Lambda$ by Lemma~\ref{lem:rowlambda}, and therefore $\alpha \in H(\cR^K,d)^{\ker\Lambda}$. In general case $\alpha \in H(\cR^K|_{\row\Lambda},d)$, write $\alpha$ as the sum of nonzero summands which are homogeneous with respect to the second grading
%    \[
%        \alpha = \alpha_1 +\dotsc+ \alpha_\ell,\quad \alpha_i \in H^{\ast,\omega_i}(\RZ_K;R)
%    \]
%    for $1 \le i \le \ell$, and $\omega_i \ne \omega_j$ if $i \ne j$. We can apply the above argument to each summand. The proof of that $H(\cR^K,d)^{\ker\Lambda} \subseteq H(\cR^K|_\Lambda,d)$ is done applying similar argument to each homogeneous summand.
{
    For $\omega \in \Z_2^m$  and a nonzero cohomology class $\alpha \in H^{\ast,\omega}(\RZ_K)$, by Lemma~\ref{lem:actiononri}, one observes that, for $g \in \ker\Lambda$,
}
    \[
        g \cdot \alpha = \left\{
                        \begin{array}{ll}
                          \alpha, & \hbox{if $g \cap \omega$ has even cardinality, or} \\
                          -\alpha, & \hbox{if $g \cap \omega$ has odd cardinality.}
                        \end{array}
                      \right.
    \]
    Note that $\alpha \ne -\alpha$; if $\alpha = -\alpha$, then $\alpha + \alpha = 2\alpha = 0$ and we conclude that $\alpha = 0$
{since $2$ is a unit in $R$.}
    Thus, by Lemma~\ref{lem:rowlambda},
    $\alpha$ is fixed by $\ker \Lambda$ if and only if $\omega \in \row \Lambda$.
    In general case $\alpha \in H(\cR^K,d)$, write $\alpha$ as the sum of nonzero summands which are homogeneous with respect to the second grading
    \[
        \alpha = \alpha_1 +\dotsc+ \alpha_\ell,\quad \alpha_i \in H^{\ast,\omega_i}(\RZ_K)
    \]
    for $1 \le i \le \ell$, and $\omega_i \ne \omega_j$ if $i \ne j$.
{
    By applying the above argument to each summand $\alpha_i$, one can see that $\alpha$ is fixed by $\ker \Lambda$ if and only if each $\omega_i$ is in $\row \Lambda$, that is, $\alpha \in H(\cR^K|_{\row\Lambda},d)$.
    This proves the theorem.
}
\end{proof}

%\blue{(Move this part from the previous section to this section?)}
The following theorem  is an application of Theorem~\ref{thm:cohoofsmallcover} and an essential part of the main theorem.

\begin{theorem}\label{thm:productondelta}
    Let $M = M^\R(K,\Lambda)$ be a real toric space. Then its cohomology group is equipped with a $(\Z\oplus \row\Lambda)$-grading induced by \eqref{eqn:bigrading} and we have
    \[
        \smile\colon H^{p,\omega}(M)\otimes H^{p',\omega'}(M) \to H^{p+p',\omega \triangle \omega'}(M).
    \]
    In other words, $H^{*,*}(M)$ is a $(\Z\oplus \row\Lambda)$-graded $R$-algebra.
%    Furthermore, \eqref{eqn:cohoiscohoofr} still holds when one replaces the relation \eqref{eqn:utrelations} by
%    \begin{multline}%\label{eqn:utrelations}
%        u_iu_i=0,\quad u_it_i=0,\quad t_iu_i=0,\quad  t_it_i = 1_R,\\
%          u_iu_j=-u_ju_i,\quad u_it_j = t_ju_i,\quad t_it_j=t_jt_i,
%    \end{multline}
%    for $i,j=1,\dotsc,m$ and $i\ne j$, for the definition of $\cR^K$.
\end{theorem}
%\begin{theorem}
%    Given the bigrading of \eqref{eqn:bigrading}, we have
%    \[
%        \smile\colon H^{p,\omega}(\RZ_K;R)\otimes H^{p',\omega'}(\RZ_K;R) \to H^{p+p',\omega \triangle \omega'}(\RZ_K;R),
%    \]
%    where $\omega \triangle \omega' = (\omega \cup \omega') \setminus (\omega \cap \omega')$ is the symmetric difference.
%\end{theorem}

%Theorem~\ref{thm:cohoofsmallcover} gives a simple proof of Theorem~\ref{thm:productondelta}, while it should be an interesting work to prove it by a direct method using \eqref{eqn:utrelations}.
\begin{proof}%[Proof of Theorem~\ref{thm:productondelta}]
    It is enough to prove the theorem when $\Lambda = 0$ and thus $M = \RZ_K$. Once it is shown for $\RZ_K$, it is instant to generalize the result for general real toric spaces, because $\row \Lambda$ is closed under the operation $\triangle$. Consider the two monomials $u_\sigma t_{\omega \setminus \sigma}$ and $u_{\sigma'}t_{\omega' \setminus \sigma'}$ where $\sigma \subseteq \omega$ and $\sigma' \subseteq \omega'$, each of which contributes to a nonzero cohomology class $\alpha \in H^{\ast,\omega}(\RZ_K)$ and $\alpha' \in H^{\ast,\omega'}(\RZ_K)$ respectively.
    We investigate how the product $h = u_{\sigma}t_{\omega \setminus \sigma}\cdot u_{\sigma'}t_{\omega' \setminus \sigma'} = \pm u_{A}t_{B\setminus A}$ is computed, when $A\subseteq B$ and $A \in K$.

    Let $i \in \omega \cup \omega'$ be a subscript possibly in $h$. We have four cases
    %\blue{(I changed the order of cases)}
    \begin{enumerate}[(a)]
        \item $i \in \omega \triangle \omega'$,
        \item $i \in \omega \cap \omega'$ and $i \notin \sigma \cup \sigma'$,
        \item $i \in \omega \cap \omega'$ and $i \in \sigma \cap \sigma'$,  and
        \item $i \in \omega \cap \omega'$ and $i \in \sigma \triangle \sigma'$.
    \end{enumerate}

    Recall the relations \eqref{eqn:utrelations}.
    If $i$ is for (a), then it contributes as $u_i$ or $t_i$ for a factor of $h$. Thus we have
    \begin{equation}\label{eqn:condition_of_B}
      \omega \triangle \omega' \subseteq B \subseteq \omega \cup \omega'.
    \end{equation}

    In order to prove the theorem, it is enough to show that $B = \omega \triangle \omega'$ if $h$ does not vanish in the cohomology. %One can easily see that the cases (b) and (c) are not meaningful as follows.
    We additionally observe that
    \begin{itemize}
        \item (b) corresponds to $t_it_i = 1$ and we obtain $i \notin B$.
        \item (c) corresponds to $u_iu_i = 0$, and therefore if this happens in $h$ then $h=0$.
        \item (d) corresponds to $u_it_i = u_i$ or $t_iu_i = -u_i$ and we obtain $u_i$ for a factor of $h$.
    \end{itemize}
    Assume that $\omega \ne \omega'$. The case $\omega = \omega'$ will be dealt with later. Observe that $\omega$ and $\omega'$ define a $\Z_2$-linear map $L\colon \Z_2^m \to \Z_2^2$ up to change of basis of $\Z_2^2$; write $\omega$ and $\omega'$ as row vectors in $\Z_2^m$ and $L$ is given by the $(m \times 2)$-matrix whose two rows are $\omega$ and $\omega'$. Then $Y = \RZ_K/ \ker L$ is a real toric space (the action of $\ker L$ need not be free) whose cohomology ring $H^\ast(Y)$ is the subring
    \begin{align*}
%        H^\ast(Y) &\cong H(\cR^K|_{\row L},d) \\
%                &= H(\cR^K_{\varnothing},d)\oplus H(\cR^K_{\omega},d) \oplus H(\cR^K_{\omega'},d) \oplus H(\cR^K_{\omega \triangle \omega'},d) \\
%                &\subseteq H(\cR^K,d) \cong H^*(\RZ_K) \\
        H^{*,\varnothing}(\RZ_K) \oplus H^{*,\omega}(\RZ_K) \oplus H^{*,\omega'}(\RZ_K) \oplus H^{*,\omega \triangle \omega'}(\RZ_K) \subseteq H^*(\RZ_K)
    \end{align*}
    and $\alpha,\alpha' \in H^\ast(Y)$. %Since the projection map $p\colon \RZ_K \to Y$ preserves the $\Z\oplus \Z_2^m$-grading, the induced map $p^* \colon H^*(Y) \to H^*(\RZ_K)$ is just the natural inclusion.
    Since $H^*(Y)$ is closed under the cup product, $B$ should be one of $\omega,\,\omega'$, and $\omega \triangle \omega'$. Suppose that $B \neq \omega \triangle \omega'$. Note that by \eqref{eqn:condition_of_B} $B = \omega \triangle \omega'$ if either $\omega \not\subseteq \omega' $ or $\omega' \not\subseteq \omega$.
%    Because $\omega \triangle \omega' \subseteq B$, we see that
%    If either of the inclusions holds, we can
    Therefore, one may assume that $\omega \subseteq \omega'$ (including the case $\omega = \omega'$) and $B = \omega'$.
%The former case is done. For the latter case,
    In this case, every $i$ in $\omega$ should be of Case (d). It means that every term of $h$ corresponding to $\omega$ is $u_i$, not $t_i$. If $h$ would not vanish in the cohomology, recall that every $u_i$ is in a face of $K$ and we observe that $\omega \subseteq A \in K$. Therefore $K_\omega$ is contractible and $[u_\sigma t_{\omega \setminus \sigma}] = 0 \in H^{*,\omega}(\RZ_K)$, which is a contradiction since we have assumed that $u_\sigma t_{\omega \setminus \sigma}$ contributes to $\alpha$.
\end{proof}
%One refers to Theorem~4.5.8 of \cite{Buchstaber-Panov2015} for an analogous result for $\cZ_K$.

\begin{remark}\label{rem:utt}
    In the proof of Theorem~\ref{thm:productondelta}, one observes that if either of (c) or (d) appears in $h$, then $h$ vanishes in the cohomology. In fact, the monomials $h$, in which (d) appears at least once, assemble to make a zero cohomology class. %, which is a sum of ``zero homogeneous classes'' in $H^{*,B}(\RZ_K;R)$ where $\omega \triangle \omega' \subsetneq B \subseteq \omega \cup \omega'.$
    Therefore one can calculate cup product of cohomology as if $u_it_i= t_iu_i=0$. This ``rule'' cannot be directly applied in place of \eqref{eqn:utrelations} since it could be problematic as
    \[
        0 = (u_it_i)t_i = u_i(t_it_i) = u_i,
    \]
    but it can be freely used to compute cup product between the summands $u_\sigma t_{\omega \setminus \sigma}$ at the cohomology level.
\end{remark}

\begin{proof}[Proof of Main Theorem]
    The proof is complete using \eqref{eqn:pomegahomology} and Theorem~\ref{thm:productondelta}.
\end{proof}

\begin{example}
    Let us consider the simplicial 2-sphere $K$ with $9$ vertices labeled $1$ to $9$ and 14 triangles
    \[
        123,129,138, 148, 149, 237,257,259, 367, 368, 456, 459, 468, \text{ and } 567.
    \]
    The sphere $K$ is the boundary of a triangular prism each of whose quadrangular faces is subdivided to four triangles respectively. We are given two cohomology classes
    \[
        \alpha = [u_5 t_{167} +  u_6 t_{157} + u_7 t_{156}] \in H^{1,1567}(\RZ_K)
    \]
    and
    \[
        \beta = [u_2 t_{347} + u_3t_{247} + u_7 t_{234}] \in H^{1,2347}(\RZ_K).
    \]
    Then the cup product $\alpha \smile \beta$, computed by the rule \eqref{eqn:utrelations}, is written as $\alpha \smile \beta = -x -y$, where
    \[
        x = [u_{25}t_{1346} + u_{36}t_{1245}] \in H^{2,123456}(\RZ_K)
    \]
    and
    \[
        y = [ u_{57} t_{12346} + u_{67} t_{12345} + u_{27}t_{13456} + u_{37} t_{12456} ] \in H^{2,1234567}(\RZ_K).
    \]
    Observe that $d( u_7t_{123456} ) = u_{57} t_{12346} + u_{67} t_{12345} + u_{27}t_{13456} + u_{37} t_{12456}$ and therefore $ y = 0$. Theorem~\ref{thm:productondelta} or the ``rule'' $u_it_i= t_iu_i=0$ in the Remark~\ref{rem:utt} implies that the calculation for $y$ is actually not needed to compute $\alpha \smile \beta$.
\end{example}

\begin{remark}
    The $\Z \oplus \Z_2^m$-grading of $H^*(\RZ_K)$ is given by
    \[
        \deg u_i = (1,\e_i) \text{ and } \deg t_i = (0,\e_i),
    \]
    where $\e_i$ is the $i$th coordinate vector of $\Z_2^m$. It is the analogue of the $\Z \oplus \Z^m$-grading of $H^*(\cZ_K)$ in Construction~3.2.8 of \cite{Buchstaber-Panov2015}.
    Recall that $H^*(\cZ_K)$ is also equipped with the famous \emph{bigrading}, which is a $\Z\oplus\Z$-grading as explained in Section~4.4 of \cite{Buchstaber-Panov2015}.
    The grading group $\Z\oplus\Z$ is a subgroup of $\Z \oplus \Z^m$, but not a subgroup of $\Z \oplus \Z_2^m$.
    The appearance of 2-torsion elements is essential in our grading due to the relation $t_it_i = 1$, and therefore the analogue of the bigrading does not behave well with the cup product for $\RZ_K$. %On the other side, the bigrading of $\cZ_K$ can be understood as a subgroup grading of the $\Z \oplus \Z^m$-grading. Suppose that $\alpha \in H^*(\cZ_K)$ be a class such that
\end{remark}

%\blue{
%A direct consequence of Theorem~\ref{thm:productondelta} is as follows. In the proof of Theorem~\ref{thm:productondelta}, $B$ cannot be $\omega \triangle \omega'$ and $h$ should vanish whenever case (c) or (d) realizes. Therefore we can improve the relation \eqref{eqn:utrelations} by
%\begin{multline}%\label{eqn:utrelations}
%    u_iu_i=0,\quad u_it_i=0,\quad t_iu_i=0,\quad  t_it_i = 1,\\
%      u_iu_j=-u_ju_i,\quad u_it_j = t_ju_i,\quad t_it_j=t_jt_i,
%\end{multline}
%for $i,j=1,\dotsc,m$ and $i\ne j$, for the computation of the cohomology ring.
%}

\section{Examples} \label{sec:examples}
\subsection{Real Bott manifolds} \label{subsec:real_Bott}
Let $K=\partial (I^n)^\ast$ be the boundary complex of the $n$-cube $I^n$.
The vertex set of $K$ is identified with $[2n]$ and the minimal non-face of $K$ consists of $\{i, n+i\}$ for all $i=1, \ldots, n$.
Let us assume that we are given a strictly upper triangular $n \times n$ matrix $A$ over the finite field $\Z_2$ whose $j$th column is $A_j$ for $1\le j \le n$.
Then, the matrix $\Lambda(A) = \left(I_n \mid I_n+A^t \right)$ represents a non-singular characteristic function over $K$, where $I_n$ is the identity $\Z_2$-matrix of size $n$.
The corresponding real toric space $M^\R(K, \Lambda(A))$ is well-known as a \emph{real Bott manifold}, and it is denoted by $M(A)$.
Indeed, it is a real toric variety, and plays an important role in toric geometry.
See \cite{CMO2017} or \cite{KM2009} for details.

The rational Betti number of a real Bott manifold has been computed in \cite[Lemma~2.1]{Ishida2011}. In this subsection, we further discuss about the rational cohomology ring of a real Bott manifold.
Now we need the notion of \emph{matroids} refering to \cite{Oxley2003}, which is an abstraction of linear dependency of vectors.

\begin{definition}
    A \emph{matroid} is the pair $T = (E,\cC)$, where $E$ is a finite set called the \emph{ground set} and $\cC$ is a set of subsets of $E$ satisfying the following axioms:
    \begin{enumerate}
        \item[\textbf{(C1)}] $\varnothing \notin \cC$.
        \item[\textbf{(C2)}] If $C_1,\;C_2\in \cC$ and $C_1 \subseteq C_2$, then $C_1 = C_2$.
        \item[\textbf{(C3)}] If $C_1,\;C_2\in \cC$ such that $C_1 \ne C_2$ and $e \in C_1 \cap C_2$, then there exists $C_3 \in \cC$ such that $C_3 \subseteq (C_1 \cup C_2) \setminus \{e\}$.
    \end{enumerate}
    The elements of $\cC$ are called the \emph{circuits} of the matroid.
\end{definition}

\begin{definition}
    Let $A$ be a matrix over the finite field $\Z_2$. Let $E = \{A_j \mid 1 \le j \le m\}$ the set of the columns of $A$ and $\cC$ the collection of subsets $C$ of $E$ for which the columns in $C$ are minimally dependent, that is, any proper subset of $C$ is linearly independent while $C$ itself is linearly dependent. Then $T = T(A)=(E,\cC)$ is called a \emph{binary matroid} and we say that $A$ \emph{represents} $T$.
\end{definition}

%A basic fact in matroid theory is that binary matroids are matroids, but the converse is not always true.

\begin{proposition} \label{prop:R-cohom_real_bott}
    For a strictly upper triangular matrix $A$ over $\Z_2$, the cohomology ring $H^\ast(M(A);R)$ depends only on the matroid $T(A)$ and is generated by the circuits of $T(A)$ as a graded $R$-algebra. More precisely, let $x_C$ be the formal symbol for the cohomology class corresponding to a circuit $C$. Then
    \[
        H^\ast(M(A);R) \cong R\langle x_C \mid C \in \cC\rangle/\sim,
    \]
    where we have the relations
    \[
      x_{C_1}x_{C_2} = \left\{
      \begin{array}{ll}
        (-1)^{|C_1|\cdot|C_2|}x_{C_2}x_{C_1}, & \hbox{if $C_1 \cap C_2 = \varnothing$;} \\
        0, & \hbox{if $C_1 \cap C_2 \ne \varnothing$.}
      \end{array}
    \right.
    \]
    The grading is given by $\deg x_C = |C|$.
\end{proposition}
\begin{proof}
    First of all, let us consider the boundary of the $n$-crosspolytope $K =   S_1^1 \star \dotsb \star  S_n^1$, where $S_i^1$ is the simplicial complex consisting of two points $x_i$ and $y_i$ and $\star$ means the simplicial join. Note that the real Bott manifold is a small cover over $K$. A nonempty induced subcomplex $K_\omega$ is homotopy equivalent to $S^{k-1}$ if and only if $\omega = \{x_{i_1},\dotsc, x_{i_k},y_{i_1},\dotsc,y_{i_k}\}$ for $1 \le i_1< \dotsb < i_k \le n$, and is null-homotopic otherwise. Observe that $x_i$ and $y_i$ correspond to the $i$th and $(n+i)$th column of $\Lambda(A)$ respectively and the proof goes obviously by the main theorem.
\end{proof}

The following easy observation characterizes the matroids $T$ which can be a matroid $T(A)$ of a strictly upper triangular matrix $A$ over $\Z_2$.

\begin{proposition}
    Let $T$ be a binary matroid which contains a singleton circuit. Then we have a strictly upper triangular matrix $A$ over $\Z_2$ representing $T$.
\end{proposition}
\begin{proof}
    Let $B$ {be} the matrix over $\Z_2$ representing $T$ and $B_j$ the $j$th column vector of $B$. By the assumption, there is a zero column, say $B_1$, after an appropriate shuffling of columns. After that, we consider the nested sequence of linear subspaces
    \[
        0 = \langle B_1\rangle\subseteq \langle B_1,B_2\rangle\subseteq \langle B_1,B_2,B_3\rangle\subseteq \dotsb
    \]
    and pick a basis $\{x_1,\dotsc,x_\ell\}$ of the column space such that $\langle B_1,\dotsc, B_{k+1} \rangle \subseteq \langle x_1,\dotsc,x_{k} \rangle$ for $1\le k \le \ell$. Then $B$ becomes a strictly upper triangular matrix with respect to this basis.
\end{proof}

\begin{remark}
    Unfortunately, the graded ring $H^\ast(M(A);R)$ does not necessarily determine the matroid $T(A)$. Let us consider the two following matroids $T_1$ and $T_2$ determined by the $\Z_2$-linear relations
    \begin{gather*}
        v_0 = 0 \\
        v_1 + v_2 + v_3+v_4+v_5+v_6+v_7+v_8+v_9+v_{10} = 0 \\
        v_3+v_4+v_5+v_6+v_7+v_8 +v_{11}+v_{12}+v_{13}+v_{14} = 0 \\
        v_5+v_6+v_7+v_8 +v_9+v_{10}+v_{13}+v_{14}+v_{15}+v_{16} = 0
    \end{gather*}
    and
    \begin{gather*}
        v_0 = 0 \\
        v_1 + v_2 + v_3+v_4+v_5+v_6+v_7+v_8+v_9+v_{10} = 0 \\
        v_2+v_3+v_4+v_5+v_6+v_7+v_{11}+v_{12}+v_{13}+v_{14} = 0 \\
        v_5+v_6+v_7+v_8 +v_9+v_{10}+v_{13}+v_{14}+v_{15}+v_{16} = 0
    \end{gather*}
    respectively. One easily checks that the two matroids are non-isomorphic. %; \blue{in $T_2$, $v_1$ is the only vector  } every circuit in $T_1$ other than $\{v_0\}$ has even cardinality, but this does not hold in $T_2$.
    On the other hand, the two matroids make isomorphic cohomology $R$-algebras, each of which is generated by one degree 1 element, three degree 8 elements, and four degree 10 elements and the products which do not involve the degree 1 element are all zero.
\end{remark}

\begin{table}
    \[\begin{array}{c|cccccccc}
        \hline
        n  & 1 & 2 & 3 & 4 & 5 & 6 & 7 & 8 \\ \hline
        \mathcal{D}_n & 1 & 2 & 4 & 12 & 54 & 472 & 8512 & 328416 \\
        \mathcal{M}_n & 1 &2 &4 & 8 & 16& 32 & 68 & 148 \\ \hline
    \end{array}\]
    %\medskip
    \caption{$\mathcal{D}_n$ is the number of diffeomorphism types of $n$-dimensional real Bott manifolds found in \cite{CMO2017}, and $\mathcal{M}_n$ is the number of non-isomorphic binary matroids on an $n$-set (A076766 of \cite{oeis}).}\label{tabl:realbottnumbers}
\end{table}

\begin{remark} \label{rem:real_Bott}
    The $\Z_2$-cohomology ring of a real Bott manifold determines its diffeomorphism type \cite{KM2009}. Therefore in Table~\ref{tabl:realbottnumbers}, $\mathcal{D}_n$ is the number of $\Z_2$-cohomology rings (up to isomorphism) of $n$-dimensional real Bott manifolds, and we know that the number of isomorphism types of $\Q$-cohomology rings of $n$-dimensional real Bott manifolds does not exceed $\mathcal{M}_n$. Since $\mathcal{M}_n < \mathcal{D}_n$ for some $n$, we conclude that the $\Q$-cohomology ring of a real Bott manifold does not determine its diffeomorphism type and thus the $\Q$-cohomology is strictly ``weaker'' than $\Z_2$-cohomology in the case of real Bott manifolds.
\end{remark}

The $\Q$-cohomology ring of the real Bott manifold is a fairly weak invariant as the above remark shows, but it is worth emphasizing that $\Q$-cohomology ring is still stronger than the $\Q$-cohomology group. Consider the two binary matroids $T_1$ and $T_2$ determined by the $\Z_2$-linear relations
\[  v_0=0,\quad v_1+v_2+v_3+v_4+v_5+v_6=0,\quad v_4+v_5+v_7=0,\quad v_5+v_6+v_8=0 \]
and
\[  v_0=0,\quad v_1+v_2+v_3+v_4+v_5+v_6=0,\quad v_3+v_4+v_7=0,\quad v_5+v_6+v_8=0 \]
respectively.
We choose two $9$-dimensional real Bott manifolds $M_1$ and $M_2$ whose corresponding binary matroids are $T_1$ and $T_2$, respectively. Then $M_1$ and $M_2$ have identical rational Betti numbers
$$
    (\beta^0, \beta^1, \ldots, \beta^9) = (1,1,0,2, 3, 3, 4, 2, 0,0).
$$ However, in $H^\ast(M_1)$, the all multiplications of elements of degree greater than 1 are trivial, while $H^\ast(M_2)$ has a non-trivial multiplication.
Hence, $M_1$ and $M_2$ have non-isomorphic $\Q$-cohomology rings, although they have isomorphic $\Q$-cohomology groups.

\subsection{Generalized real Bott manifolds} \label{subsec:gen_real_bott}
Let us consider a more generalized notion of real Bott manifolds.
Let $K = \partial ( \prod_{i=1}^k \Delta^{n_i} )^\ast $ be the boundary complex of the product of simplices $\prod_{i=1}^k \Delta^{n_i}$.
Then, the vertex set of $K$ is $\{1_1, \ldots, 1_{n_1+1}, 2_1, \ldots, 2_{n_2 + 1}, \ldots, k_1, \ldots, k_{n_k +1} \}$, and the minimal non-face of $K$ consists of $\{ i_1, \ldots, i_{n_i +1 } \}$ for all $i=1, \ldots, k$.

Let us assume that we are given a $k \times k$ block matrix $\bA$ over $\Z_2$ which is strictly upper triangular whose $(i,j)$th block of $A$ is of size $1 \times n_i$.
We denote by $\bI$ the $k\times k$ block matrix whose diagonal elements are all $1$ and the others are all $0$, where the size of block is equal to that of $A$.
Put $n = n_1 + \cdots + n_k$ and $m = n + k$.
Then, the $n\times m$ matrix $\Lambda(\bA) = \left(I_n \mid \bI^t +\bA^t \right)$ represents a non-singular characteristic function over $K$, where each column of $\Lambda(\bA)$ is assigned by the vertex set of $K$ in the order of
$$
\{1_1, \ldots, 1_{n_1}, 2_1, \ldots, 2_{n_2},\ldots, k_1, \ldots, k_{n_k} \mid 1_{n_1+1}, 2_{n_2+1}, \ldots, k_{n_k+1}\}.
$$
The corresponding real toric space $M^\R(K, \Lambda(\bA))$ is a ($k$-stage) \emph{generalized real Bott manifold}, and it is denoted by $M(\bA)$.
Let $A$ be the $k \times k$ matrix over $\Z_2$ such that the $(i,j)$-component of $A$ is congruent to the sum of all components of the $(i,j)$th block of $\bA$ if $i \neq j$ or $n_i +1$ if $i=j$. We call this $A$ by the \emph{underlying matrix} of $M(\bA)$.
One remarks that if $n_1 = \cdots = n_k = 1$, then $M(\bA)$ is indeed a real Bott manifold $M(A)$.
See \cite{Choi-Masuda-Suh2010} for details.

Then, similarly to Proposition~\ref{prop:R-cohom_real_bott}, we have the following proposition.
\begin{proposition} \label{prop:R-cohom_gen_real_bott}
    The cohomology ring $H^\ast(M(\bA);R)$ depends only on the matroid $T(A)$ and the integers $n_1, \ldots, n_k$. It is generated by the circuits of $T(A)$ as a graded $R$-algebra. More precisely, let $x_C$ be the formal symbol for the cohomology class corresponding to a circuit $C$. Then
    \[
        H^\ast(M(A);R) \cong R\langle x_C \mid C \in \cC\rangle/\sim,
    \]
    where we have the relations
    \[
      x_{C_1}x_{C_2} = \left\{
      \begin{array}{ll}
        (-1)^{|C_1|\cdot|C_2|}x_{C_2}x_{C_1}, & \hbox{if $C_1 \cap C_2 = \varnothing$;} \\
        0, & \hbox{if $C_1 \cap C_2 \ne \varnothing$.}
      \end{array}
    \right.
    \]
    The grading is given by $\deg x_C = \sum_{i \in C} n_i$.
\end{proposition}

\begin{corollary} \label{cor:2-stage}
  The $\Z_2$-cohomology ring of a two-stage generalized real Bott manifold determines its $\Q$-cohomology ring.
\end{corollary}
\begin{proof}
    Assume that $M(\bA)$ and $M(\bB)$ have the same $\Z_2$-cohomology rings.
    Let $A$ and $B$ be the underlying matrices of $\bA$ and $\bB$, respectively.
    Let $p$ (resp., $q$) be the number $q$ of non-zero components of the $(1,2)$th block of $\bA$ (resp., $\bB$).
    Then, by \cite[Theorem~2.3]{Masuda2010}, $p \equiv q \text{ or } n_2 + 1 - q (\text{mod $2^{h(n_1+1)}$})$, where $h(a)$ is the minimal integer $r$ such that $2^r \geq a$. Since $n_1 \geq 1$, $h(n_1+1)$ is positive, so $p$ must have the same parity with $q$ and $n_2+1 -q$.
    If $n_2$ is odd, then $A$ and $B$ are the same. Hence, $M(\bA)$ and $M(\bB)$ have the isomorphic $\Q$-cohomology rings.

    If $n_2$ is even, then $A$ and $B$ can be different. However, in this case, $A$ and $B$ must be of form
    $$A = \begin{pmatrix}
      a & b \\ 0 & 1
    \end{pmatrix} \quad \text{ and } \quad B = \begin{pmatrix}
      a & c \\ 0 & 1
    \end{pmatrix},$$
    where $a, b, c \in \Z_2$.
    Note that $A$ and $B$ have the same first column, and the second column never correspond to the generator of the $\Q$-cohomology ring by Proposition~\ref{prop:R-cohom_gen_real_bott} because there is no dependent column set containing the second column.
    Therefore, $M(\bA)$ and $M(\bB)$ have still isomorphic $\Q$-cohomology rings, as desired.
\end{proof}

There was one interesting question in toric topology called the \emph{cohomological rigiditiy problem for small covers} \cite[Section~4]{Choi-Masuda-Suh2011}: if two small covers have the isomorphic $\Z_2$-cohomology rings, then are they diffeomorphic?
As we mentioned in Remark~\ref{rem:real_Bott}, it is positive for real Bott manifolds.
It, however, is not true in general.
Masuda \cite{Masuda2010} showed that two-stage generalized Bott manifolds provide counterexamples to the problem.
Nevertheless, motivated by Corollary~\ref{cor:2-stage}, it is reasonable to ask the following weaker version of the cohomological rigidity problem.

\begin{question}
  Does the $\Z_2$-cohomology ring of a small cover determine its rational cohomology ring?
\end{question}

\begin{remark}
  It is shown in \cite[Lemma~8.1]{Choi-Masuda-Suh2010TAMS} that every $\Z_2$-cohomology ring isomorphism preserves the Stifel-Whitney class of a small cover. Therefore, the $\Z_2$-cohomology ring of a small cover determines its orientability, and, thus, its $n$th rational cohomology group as well.
  This fact also supports an affirmative evidence of the above question.
\end{remark}

\subsection{Cohomologically symplectic real toric spaces} \label{subsec:cohom_symp}
In this subsection, we assume that $K = \partial P^*$ is the boundary of a simplicial polytope $P^*$ and therefore $\RZ_K$ is a smooth manifold.

\begin{definition}
    A closed manifold $M$ of dimension $2n$ is called \emph{cohomologically symplectic} or \emph{c-symplectic} if there is a cohomology class $\alpha \in H^2(M;\R)$ such that $\alpha^n \ne 0$.
\end{definition}

\begin{lemma}\label{lem:cohosymrealtoricspace}
    Let $K$ be the boundary of a simplicial $2n$-polytope with $m$ vertices and $\Lambda$ a characteristic function over $K$. Then the following hold.
    \begin{enumerate}
        \item The real moment-angle manifold $\RZ_K$ is cohomologically symplectic if and only if there are $n$ homogeneous classes $\alpha_i \in H^{2,\omega_i}(\RZ_K;\Q)$ for $1 \le i \le n$ such that $\alpha_1 \smile \dotsb \smile \alpha_n \ne 0$.
        \item The real toric space $M^\R(K,\Lambda)$ is cohomologically symplectic if and only if there are $n$ homogeneous classes $\alpha_i \in H^{2,\omega_i}(\RZ_K;\Q)$ for $1 \le i \le n$ such that $\omega_i \in \row \Lambda$ for all $i$ and $\alpha_1 \smile \dotsb \smile \alpha_n \ne 0$.
    \end{enumerate}
    In either case, one must have $\omega_1 \triangle \dotsb \triangle \omega_n = [m]$.
\end{lemma}
\begin{proof}
    First, we give a proof of (1). For ``if'' part, we put $\alpha = \alpha_1 + \dotsb + \alpha_n$. Then $\alpha^n = n! \cdot \alpha_1\smile \dotsb \smile \alpha_n \ne 0$. For ``only if'' part, write $\alpha$ as a sum of homogeneous classes in $H^{2,\omega_i}(\RZ_K;\Q)$, that is, $\alpha = \beta_1 + \dotsb + \beta_\ell$ and take the power of $n$. Then in the expansion of $(\beta_1 + \dotsb + \beta_\ell)^n$, there should exist a nonzero monomial in $H^{2n}(\RZ_K;\Q) = H^{2n,[m]}(\RZ_K;\Q)$ and the proof (1) is done. The proof of (2) is just an analogue of that of (1) together with Theorem~\ref{thm:cohoofsmallcover}.
\end{proof}
\begin{definition}
    Let $K$ be the boundary of a simplicial polytope with $m$ vertices. Suppose that there are $\ell$ homogeneous classes $\alpha_1,\dotsc,\alpha_\ell \in H^{d,\omega_i}(\RZ_K;\Q)$ for $d=1$ or $2$, such that $\alpha_1 \smile \dotsb \smile \alpha_\ell \ne 0$ and $\omega_1 \triangle \dotsb \triangle \omega_\ell = [m]$. When $\balpha = \{\alpha_1,\dotsc,\alpha_\ell\}$, we say that $K$ is \emph{almost cohomologically symplectic with class set $\balpha$}, or shortly \emph{almost c-symplectic}.
\end{definition}

Note that $\cZ_K$ is never c-symplectic since it is 2-connected (Proposition~4.3.5 of \cite{Buchstaber-Panov2015}). In contrast, there are infinitely many examples of c-symplectic real moment angle manifolds as seen below. When $K = \partial P^*$ is almost c-symplectic, observe that $\RZ_K$ is c-symplectic if and only if $P^*$ has even dimension by Lemma~\ref{lem:cohosymrealtoricspace}. Let us denote by $V(K)$ the set of vertices of $K$. Recall that $K$ is called \emph{flag} if every non-face $I \subseteq V(K)$ contains a non-face of cardinality two.

\begin{proposition}\label{prop:flagiscsymp}
    Let $K$ be the boundary of a simplicial polytope. If $K$ is flag, then it is almost c-symplectic.
\end{proposition}
\begin{proof}
    Let us assume that $K$ is of dimension $n-1$ and pick a facet $\{v_1,\dotsc,v_n\}$ of $K$. Since $K$ is a pseudomanifold, the link of the codimension two face
    \[
        \{v_1,\dotsc,\widehat{v_i},\dotsc,v_n\} =  \{v_1,\dotsc,v_{i-1},v_{i+1},\dotsc,v_n\},
    \]
    for $1\le i \le n$, is a set of two elements one of which is $v_i$ and the other is denoted by $w_i$. Observe that $w_i \ne w_j$ if $i \ne j$ thanks to flagness. We take a vertex labeling $\ell\colon V(K)\setminus\{v_1,\dotsc,v_n\} \to \{1,\dotsc,n\}$ such that $\ell(w_i) = i$ for all $i$, and $\ell(v) \ne i$ if $v$ is connected to $v_i$ by an edge. The map $\ell$ exists (not necessarily uniquely) because there is no vertex $v$ connected to $v_i$ by an edge for all $i$, again by flagness. Now the cohomology classes $\alpha_i = [u_{v_i}\prod_{v:\ell(v) = i}t_v]$ are well-defined and one checks that $\alpha_1\smile \dotsb \smile \alpha_n \ne 0$, completing the proof.
\end{proof}

The converse of the above Proposition does not hold. For example, let us denote by $\partial P_k$ the boundary of the $k$-gon. Then the simplicial join $K = \partial P_{k_1} \star \dotsb \star P_{k_\ell}$ is flag if and only if $k_i \ge 4$ for all $i$. But $K$ is always almost c-symplectic as one can see below.

\begin{proposition}
    Let $K$ and $L$ are the boundaries of two simplicial polytopes. Then the simplicial join $K \star L$ is almost c-symplectic if and only if both of $K$ and $L$ are almost c-symplectic.
\end{proposition}
\begin{proof}
    One direction is obvious since $\RZ_{K \star L} = \RZ_K \times \RZ_L$. We show the other direction. Let $\alpha \in H^{*,\omega \sqcup \tau}(\RZ_{K \star L};\Q)$ where $\omega \subseteq V(K)$ and $\tau \subseteq V(L)$. One applies K\"{u}nneth theorem to $K_{\omega \sqcup \tau} = K_\omega \star K_\tau$ and
    \[
        \widetilde{H}^{*,\omega \sqcup \tau}(\RZ_{K \star L};\Q) \cong \widetilde{H}^{\ast,\omega}(\RZ_K;\Q) \otimes \widetilde{H}^{\ast,\tau}(\RZ_L;\Q)
    \]
    and therefore $\alpha$ is a sum of classes of the form $\beta \smile \gamma$, where $\beta$ and $\gamma$ are homogeneous classes in $H^{\ast,\omega}(\RZ_K;\Q)$ and $H^{\ast,\tau}(\RZ_L;\Q)$ respectively. In particular, the proposition is proved putting $\omega = V(K)$ and $\tau = V(L)$. %Now the proof is done similarly to that of Lemma~\ref{lem:cohosymrealtoricspace}.
\end{proof}

The previous Proposition can again be applied for real toric spaces. Consider the simplicial join $K = K_1 \star \dotsb \star K_k$ and suppose that $\Lambda$ is a characteristic function over $K$. Suppose that $M^\R(K,\Lambda)$ is cohomologically symplectic. Then one can choose the classes $\alpha_i^j \in H^{d,\omega_i^j}(\RZ_{K_j};\Q)$ for $d=1$ or $2$, $\omega_i^j \subseteq V(K_j)$,  $1\le j \le k$ and $1 \le i \le \ell_j$ such that
\[
    \alpha_1^j \smile \dotsb \smile \alpha_{\ell_j}^j \ne 0,\quad \omega_1^j \triangle \dotsb \triangle \omega_{\ell_j}^j = V(K_j).
\]
Denote by $\balpha^j = \{\alpha_1^j,\dotsc,\alpha_{\ell_j}^j\}$ a set of cohomology classes. Write $\balpha = \balpha^1 \cup \dotsb \cup \balpha^k$ and $\balpha_d = \balpha \cap H^d(\RZ_K;\Q)$ for $d=1,2$.  Suppose that for each $\alpha \in \balpha_2$, $\omega \in \row \Lambda$ if $\alpha\in H^{2,\omega}(\RZ_K;\Q)$, and there is a bijection $\phi\colon \balpha_1 \to \balpha_1$ called a \emph{pairing map} such that
\begin{enumerate}
    \item $\phi(\phi(\alpha)) = \alpha$ for all $\alpha\in A_1$,
    \item $\phi(\alpha) \ne \alpha$, and
    \item $\omega \triangle \omega' \in \row \Lambda$ whenever $\alpha \in H^{1,\omega}(\RZ_K;\Q)$ and $\phi(\alpha) \in H^{1,\omega'}(\RZ_K;\Q)$.
\end{enumerate}
If this assumption holds then we say that $\balpha$ is \emph{$\Lambda$-compatible}. We omit the proof of the below proposition which generalize a result in \cite{Ishida2011}.

\begin{proposition}
    For a characteristic function $\Lambda$ over $K_1 \star \dotsb \star K_k$, the real toric space $M^\R(K_1 \star \dotsb \star K_k,\Lambda)$ is c-symplectic if and only if each $K_i$ is c-symplectic with class set $\balpha_i$ such that $\balpha_1,\dotsc,\balpha_k$ are $\Lambda$-compatible.
\end{proposition}

\begin{remark}
    Recall that a closed 2-form $\omega$ in a smooth $2n$-manifold is a \emph{symplectic form} if $\omega^n$ is nowhere vanishing. A naturally following question is: ``For what $K$ does the real moment-angle manifold $\RZ_K$ admit a symplectic form?'' The general answer seems quite non-trivial and the only known examples are the simplicial joins of polygon boundaries whose corresponding real moment-angle manifolds are products of orientable surfaces. It is worthwhile to note that in \cite{GL2013}, Gitler and Lopez de~Medrano presented many families of $\RZ_K$ diffeomorphic to connected sums of sphere products, but the sphere products almost always contain a sphere factor of dimension $\ge 3$, preventing the manifold from being symplectic.
\end{remark}
We present the following two questions related to the above Remark.
\begin{question}
    Let $K = \partial P^*$ be the boundary of a simplicial $2n$-polytope.
    \begin{enumerate}
        \item If $K$ is almost c-symplectic, then does $\RZ_K$ admit a symplectic form?
        \item If $K$ is flag, then does $\RZ_K$ admit a symplectic form?
    \end{enumerate}
\end{question}
In Proposition~\ref{prop:flagiscsymp}, a cohomology class of $\RZ_K$ of top degree is generated by degree one classes. The following can be regarded as its strengthening.
\begin{conjecture}
    Let $K = \partial P^*$ be the boundary of a simplicial $2n$-polytope. Then $K$ is flag if and only if $H^*(\RZ_K;\Q)$ is generated by degree one elements.
\end{conjecture}

\section*{Acknowledgments}
The authors would like to appreciate to Dr. Li~Cai for the motivating communication concerning the cup product.
They are also thankful to Professors Hiroaki Ishida and Shizuo Kaji for helpful comments about the transfer homomorphism.
The first named author thanks to Professor Alex Suciu at Northeastern university for providing a wonderful sabbatical location conducive to thinking deep thoughts.
The second named author is grateful to Professor Yunhyung Cho for the discussion about cohomologically symplectic real moment angle manifolds.

\bigskip

\providecommand{\bysame}{\leavevmode\hbox to3em{\hrulefill}\thinspace}
\providecommand{\MR}{\relax\ifhmode\unskip\space\fi MR }
\providecommand{\MRhref}[2]{%
  \href{http://www.ams.org/mathscinet-getitem?mr=#1}{#2}
}
\providecommand{\href}[2]{#2}

\end{document}